\documentclass[a4,11pt]{amsart}
\usepackage{amsfonts}
\usepackage{mathrsfs}
\usepackage{amsthm,amsxtra}
\usepackage{amssymb}
\usepackage{amsmath}
\usepackage{amscd}
\usepackage{enumerate}
\usepackage{hyperref}
\usepackage{comment}
\usepackage{cite}
\usepackage{t1enc}
\usepackage[mathscr]{eucal}
\usepackage{indentfirst}
\usepackage{graphicx}
\graphicspath{ {images/} }
\usepackage{xcolor}
\usepackage[margin=2.9cm]{geometry}
\usepackage{tikz-cd}
\usepackage{tikz}

\definecolor{green}{HTML}{37cc57}
\definecolor{red}{HTML}{e81e5b}
\definecolor{blue}{HTML}{1233b5}

\theoremstyle{plain}

\newtheorem{theorem}{Theorem}[section]
\newtheorem{proposition}[theorem]{Proposition}
\newtheorem{lemma}[theorem]{Lemma}
\newtheorem{corollary}[theorem]{Corollary}

\theoremstyle{definition}

\theoremstyle{remark}

\DeclareMathOperator{\Sym}{Sym}

\DeclareMathOperator{\Inn}{Inn}
\DeclareMathOperator{\Aut}{Aut}

\DeclareMathOperator{\Out}{Out}

\DeclareMathOperator{\PSL}{PSL}

\DeclareMathOperator{\Sz}{Sz}
\DeclareMathOperator{\OO}{\mathcal{O}}
\DeclareMathOperator{\GG}{\mathcal{G}}

\DeclareMathOperator{\PG}{\mathbf{P}\mathbf{G}}

\DeclareMathOperator{\ZZ}{\mathbb{Z}}
\DeclareMathOperator{\PGammaL}{P\Gamma L}
\DeclareMathOperator{\PGL}{PGL}
\DeclareMathOperator{\Gal}{Gal}

\makeatletter
\DeclareFontFamily{OMX}{MnSymbolE}{}
\DeclareSymbolFont{MnLargeSymbols}{OMX}{MnSymbolE}{m}{n}
\SetSymbolFont{MnLargeSymbols}{bold}{OMX}{MnSymbolE}{b}{n}
\DeclareFontShape{OMX}{MnSymbolE}{m}{n}{
    <-6>  MnSymbolE5
   <6-7>  MnSymbolE6
   <7-8>  MnSymbolE7
   <8-9>  MnSymbolE8
   <9-10> MnSymbolE9
  <10-12> MnSymbolE10
  <12->   MnSymbolE12
}{}
\DeclareFontShape{OMX}{MnSymbolE}{b}{n}{
    <-6>  MnSymbolE-Bold5
   <6-7>  MnSymbolE-Bold6
   <7-8>  MnSymbolE-Bold7
   <8-9>  MnSymbolE-Bold8
   <9-10> MnSymbolE-Bold9
  <10-12> MnSymbolE-Bold10
  <12->   MnSymbolE-Bold12
}{}

\let\llangle\@undefined
\let\rrangle\@undefined
\DeclareMathDelimiter{\llangle}{\mathopen}%
                     {MnLargeSymbols}{'164}{MnLargeSymbols}{'164}
\DeclareMathDelimiter{\rrangle}{\mathclose}%
                     {MnLargeSymbols}{'171}{MnLargeSymbols}{'171}
\makeatother

\title[Groups represented by incidence geometries]{Groups represented by incidence geometries}

\author{Dimitri Leemans}\thanks{This research was made possible thanks to an Action de Recherche Concert\'ee grant from the Communaut\'e Fran\c caise Wallonie-Bruxelles.}
\address{Dimitri Leemans, Universit\'e Libre de Bruxelles, D\'epartement de Math\'ematique, C.P.216 - Alg\`ebre et Combinatoire, Boulevard du Triomphe, 1050 Brussels, Belgium, Orcid number 0000-0002-4439-502X.}
\curraddr{}
\email{leemans.dimitri@ulb.be}
\urladdr{}

\author{Klara Stokes}
\address{Klara Stokes, Department of Mathematics and Mathematical Statistics, Ume\aa\; University,
901 87 Ume\aa, Sweden, Orcid number 0000-0002-5040-2089.}
\email{klara.stokes@umu.se}

\author{Philippe Tranchida}
\address{Philippe Tranchida, Max Planck Institute for Mathematics in the Sciences, Inselstrasse 22 D-04107 Leipzig, Orcid number 0000-0003-0744-4934.}
\curraddr{}
\email{tranchida.philippe@gmail.com}
\urladdr{}

\date{\today}
\subjclass{}{}
\begin{document}

\maketitle
\begin{abstract}
The aim of this paper is to use the framework of incidence geometry to develop a theory that permits to model both the inner and outer automorphisms of a group $G$ simultaneously. More precisely, to any group $G$, we attempt to associate an incidence system $\Gamma$ whose group of type-preserving automorphisms is $\Inn(G)$, the group of inner automorphisms of $G$, and whose full group of automorphisms is the group $\Out(G)$ of outer automorphisms of $G$, getting what we call an incidence geometric representation theory for groups. 
Hence, in this setting, the group $\Inn(G)$ preserves the types of $\Gamma$ while the group $\Out(G)$ is acting non-trivially on the typeset of $\Gamma$, realizing the outer automorphisms of $G$ as correlations of the incidence system.
We give examples of incidence geometric representations for the dihedral groups, the symmetric groups, the automorphism groups of the five platonic solids, families of classical groups defined over fields such as the projective linear groups, and finally subgroups of free groups $F_n$ whose outer automorphism group is the largest finite subgroup of $\Aut(F_n)$.
\end{abstract}

\section{Introduction}

A central challenge in group theory is to understand the structure of a given group. This is not only important from the viewpoint of group theory but also fundamental in various applications in algebra, geometry, and combinatorics. An approach that has shown to be very successful is to let the group act upon a geometric or combinatorial structure. After carefully specifying certain desired properties of the group action, the structure upon which the group acts can be used to study the group itself. This is the idea motivating the development of the theory of buildings and diagram geometry, developed by Jacques Tits and others~\cite{BC2013,buek,buekenhout2013diagram,Tits1982,Tits1974}.

In this work, we extend this principle by constructing incidence systems that do not only serve as representations of groups but also of their outer automorphism groups. Specifically, for a given group $G$, we build incidence systems such that the group of type-preserving automorphisms is isomorphic to the inner automorphism group of $G$ and the correlation group aligns with the group of automorphisms of $G$. We hope that our construction will reveal deep connections between the algebraic properties of groups and the combinatorial structures of geometries, providing new tools for exploration and analysis. Unlike traditional representations of groups, our approach constructs incidence systems that encapsulate both the inner and outer automorphism groups of $G$, providing a unified framework for analyzing these algebraic properties geometrically. 

Classical examples of geometries with correlation groups related to the automorphism group of the symmetry group are the projective geometries (in the language of buildings the $A_n$-geometries) in which the dualities act as correlations of the geometry. Less known but more intriguing is the incidence building defined from the points, the lines and the two classes of $3$-spaces (“Greek” and “Latin”) contained in the $6$-dimensional projective hyperbolic quadric known as the Study quadric or the $D_4$-geometry \cite{Cartan,Study1,Study2,tits1959trialite}. Study investigated this geometry in the late 19th century motivated by the fact that its points parametrize the rigid motions in a three-dimensional Euclidean space. The Study quadric features a triality, a correlation of order $3$ that permutes the points with the Greek and the Latin $3$-spaces. Dualities and trialites are fundamental in geometry; by swapping the role of the corresponding geometric objects a duality makes $|S_2|=2$ theorems out of every projective theorem. If we then also have a triality, we get $|S_3|=3!=6$ theorems out of every theorems.

In a series of recent papers, the authors have shown the existence of incidence geometries with trialities representing the outer automorphism group of the type-preserving automorphism group of the geometry \cite{LeemansStokes2019,leemans2022incidence,leemans2024, LST2025}. Building on these preliminary works, we now introduce the notion of incidence geometric representations, which further generalizes these ideas by focusing on incidence systems where the outer automorphism group acts as correlations and the inner automorphism group acts as type-preserving automorphism group.

More precisely, let $\Gamma$ be an incidence system, $\Aut(\Gamma)$ be the group of all incidence preserving bijections of $\Gamma$ and $\Aut_I(\Gamma)$ be the normal subgroup of $\Aut(\Gamma)$ consisting of type preserving bijections. 
An \textit{incidence geometric representation for a pair $(H,K)$ of groups}, with $H < K$, is an incidence system $\Gamma$ together with a pair of isomorphisms $\varphi_1 \colon H \to \Aut_I(\Gamma)$ and $\varphi_2 \colon K \to \Aut(\Gamma)$ such that the following diagram commutes

\begin{center}
\begin{tikzcd}
K \arrow[r,"\varphi_2"]                              & \Aut(\Gamma)                   \\
H \arrow[u, hook] \arrow[r, "\varphi_1"] & \Aut_I(\Gamma) \arrow[u, hook]
\end{tikzcd}
\end{center}
where the vertical arrows are the natural injections. Notice that $\Aut_I(\Gamma)$ is always a normal subgroup of $\Aut(\Gamma)$. Hence, for an incidence geometric representation to exist for a pair $(H,K)$, we must also have that $H$ is a normal subgroup of $K$.

An \textit{incidence geometric representation for a group $G$} is then an incidence geometric representation for the pair $(\Inn(G), \Aut(G))$.
If the maps $\varphi_1$ and $\varphi_2$ are not isomorphisms but are only injective, we will say that the incidence geometric representation is \textit{weak}.
By abuse of notation, when the isomorphisms $
\varphi_1$ and $\varphi_2$ are clear from context, we will say that incidence system $\Gamma$ is the (weak) incidence geometric representation for $(H,K)$ or for $G$.

With the purpose of illustrating the versatility of the model we propose, we provide a list of examples of incidence geometric representations for both finite and infinite groups: the symmetric groups, the dihedral groups, the automorphism groups of the $5$ platonic solids, some classical groups and certain subgroups of the free group $F_n$.  Outer space is a space upon which the outer automorphism group of $F_n$ acts \cite{Vogtmann2002}, and we hope that our efforts to bring tools from incidence geometry into the study of the outer automorphism group of $F_n$ will shed more light upon the structure of Outer space.



\section{Preliminaries}

In this section, we define the various objects and properties that will be used in the rest of the paper.

An {\em incidence system} is a quadruple $\Gamma=(X,*,t,I)$ where $X$ is a set of {\em elements}, $I$ is a finite set of {\em types}, $t$ is a surjective {\em type function} $t:X\twoheadrightarrow I$ and $*$ is an {\em incidence relation}, that is, a binary relation on the set $X$, such that elements of the same type are not related.
The cardinality $|I|$ is called the {\em rank} of $\Gamma$. 
The relation $*$ is called the {\em incidence relation} of $\Gamma$. 
As any symmetric binary relation, it can be represented using a symmetric matrix or an undirected graph. The graph representing the incidence relation is called the {\em incidence graph} of the incidence system. It is an $n$-partite graph (with $n := |I|$), with the elements of each type in each partition. 
 
A {\em flag} of an incidence system is a set of pairwise incident elements. The {\em rank} of a flag is the number of types in the flag. Since all elements in a flag have different type, this coincides with the number of elements in the flag. The type of a flag $F$ is the set $t(F)$.
A {\em chamber} is a flag of type $I$.
An incidence system  is an {\em (incidence) geometry} if every flag is contained in a chamber. 
A geometry is {\em firm} if every non-maximal flag is contained in at least two chambers.

The {\em residue} of a flag $F$ in a geometry $\Gamma$ is the set of elements in $\Gamma$ that are not in $F$ but that are incident with all elements of $F$. The residue of a flag in a geometry is a geometry. 
If $\Gamma$ is a geometry, then every flag is contained in a chamber, therefore the residue of a flag of rank $k$ in a geometry of rank $n$ is $n-k$. 
We say that an incidence geometry $\Gamma$ is {\em residually connected} if the incidence graphs of all its residues of rank at least two are connected graphs.
 
If $J$ is a flag of the geometry $\Gamma=(X,*,t,I)$, then the {\em $J$-truncation} of $\Gamma$, denoted by $^J\Gamma$, 
is the geometry $(X',*',t',J)$, whose set of elements $X'$ consists of the elements of $\Gamma$ with type in $J$, and the type function $t'$ and the incidence relation $*'$ are the restriction of $t$ and $*$ to $X'$ respectively. 

A {\em correlation} of an incidence geometry $\Gamma(X,*,t,I)$ is a permutation $\alpha$ of $X$ such that for all $x,y\in X$, $t(x) = t(y) \iff t(\alpha(x)) = t(\alpha(y))$ and $x*y \iff \alpha(x) * \alpha(y)$.
The set of all correlations of a geometry $\Gamma$, together with composition, forms a group called the {\em correlation group} of $\Gamma$ and denoted by ${\rm Aut}(\Gamma)$.
It contains a subgroup consisting of all the elements that induce the identity on the set of types (that is, correlations $\alpha$ such that $t(x) = t(\alpha(x))$ for all $x\in X$). This subgroup is called the group of {\em automorphisms} of $\Gamma$ and it is a normal subgroup of ${\rm Aut}(\Gamma)$ that we denote by ${\rm Aut}_I(\Gamma)$. 

Given an incidence geometry $\Gamma$ and a subgroup $G\leq Aut_I(\Gamma)$, we say that $G$ is {\em flag-transitive} on $\Gamma$ if the action of $G$ is transitive on the set of flags of type $J$ for every $J\subseteq I$. If there exists $G\leq \Aut_I(\Gamma)$ such that $G$ is flag-transitive on $\Gamma$ we say that $\Gamma$ is {\em flag-transitive}.

\section{Dihedral groups}
The first example that we will investigate is the one of dihedral groups, denoted by $D_{2n}$ for some integer $n$. The group $D_{2n}$ is the automorphism group of an $n$-sided polygon $P_n$ and has cardinality $2n$. Let $\phi(n)$ be the Euler Totient function and $\tau(n)$ be the set of totatives of $n$ (i.e: the set of integers between $1$ and $n-1$ that are coprime to $n$). We have that $\Aut(D_{2n}) \cong \text{Hol}(n) \cong C_n \rtimes C_{\phi(n)}$ (see~\cite[Page 224]{Miller1908}). The center $Z(D_{2n})$ of $D_{2n}$ is trivial if $n$ is odd and is isomorphic to $C_2$ if $n$ is even.
If an integer $k$ is in $\tau(n)$, then its inverse $-k \pmod n$ in $\ZZ_n$ is also in $\tau(n)$. Let $\Bar{\tau}(n)$ be the set that contains the smallest of $k$ and $-k \pmod n$ for each $k \in \tau(n)$. Note that $|\Bar{\tau}(n)| = \frac{\phi(n)}{2}$.

Before defining an incidence system $\Gamma(D_{2n})$ for $D_{2n}$, we need to introduce some terminology.
The {\em $n$-sided polygon} $P_n$ can be seen as a graph with $n$ vertices and $n$ edges which form a cycle. Let us denote the set of vertices of $P_n$ by $V$ and the set of edges of $P_n$ by $E_1$. For any vertex $v 
\in V$ and for any integer $1 \leq k \leq n-1$, we can play the following game. Start with $v$ and then connect it to a vertex $w$ that is at distance $k$ from $v$ in $P_n$ (there are exactly two choices for $w$). Then repeat the process starting from $w$. There are two possible outcomes: either you visit all vertices of $P_n$ by this process and thus produce a Hamiltonian cycle, or you only visit a fraction of all the vertices. The first result occurs when $k$ is co-prime with $n$ so that $k \in \tau(n)$. In this case, the Hamiltonian cycle forms a regular inscribed polygon with edges of length $k$.  The second case happens when $k$ and $n$ are not co-prime, in which case you will visit exactly $n/(n,k)$ vertices of $P_n$. For any $k \in \tau(n)$, let us denote by $E_k$ the set of edges (unordered pairs of vertices) appearing in the Hamiltonian path produces by the process described above. Note that $E_k$ does not depend on the choice of starting vertex $v$ and that $E_k = E_{(-k \pmod n)}$. 

We divide the construction of $\Gamma(D_{2n})$ into two cases, according to whether $n$ is odd or even. See Figures \ref{D10} and \ref{D16} for an example of each case.

\subsection{$n$ is odd} 
We define an incidence system $\Gamma = \Gamma(D_{2n})=(X = V \sqcup (\sqcup_{i \in \Bar{\tau}(n)}E_i),t,*)$ over $I = \{0\} \cup \Bar{\tau}(n)$ as follows:
\begin{itemize}
    \item $V$ is the set of vertices of $P_n$ and, for $i\in \Bar{\tau}(n)$, $E_i$ is the set of edges of an inscribed regular polygon with edges of length $k$;
    \item $t(X_i) = i$;
    \item Let $v \in X_0$ and $e \in X_i$ with $i \in \Bar{\tau}(n)$, then $e = \{v_1,v_2\}$ with $v_1,v_2 \in V$ and we set $v * e$ if and only if $v = v_1$ or $v = v_2$;
    \item For $i,j \in I \setminus \{0\}$, we set that every elements $e_i\in X_i$ and $e_j \in X_j$ are incident as long as $i \neq j$.
\end{itemize}


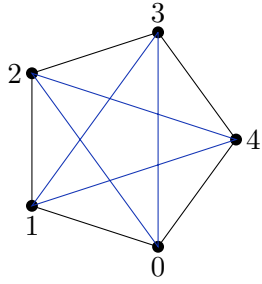
\begin{figure}
\begin{center}
\begin{tikzpicture}[scale = 1.5]
\label{D10}

\foreach \i in {1,...,5}
    \fill (\i*360/5:1) coordinate (5\i) circle(1.5 pt);
\draw (51) -- (52) -- (53) -- (54) -- (55) -- (51) ;
\draw [blue] (51) -- (53) -- (55) -- (52)-- (54) -- (51);
\filldraw[black] (54) circle (0pt)  node[anchor=north]{$0$};
\filldraw[black] (53) circle (0pt)  node[anchor=north]{$1$};
\filldraw[black] (52) circle (0pt)  node[anchor=east]{$2$};
\filldraw[black] (51) circle (0pt)  node[anchor=south]{$3$};
\filldraw[black] (55) circle (0pt)  node[anchor=west]{$4$};

\end{tikzpicture}
\caption{The geometry $\Gamma(D_{10})$. Since $n =5$ is a prime number, the underlying graph for $\Gamma(D_{10})$ is the complete graph.}   
\end{center}
\end{figure}

When $n$ is even, the vertex set $V$ of $P_n$ can be partitioned into two sets $V_0$ and $V_{-1}$ such no two vertices in $V_0$ are incident and no two vertices in $V_{-1}$ are incident.

\subsection{$n$ is even}
We define an incidence system $\Gamma = \Gamma(D_{2n})=(X = V_{-1} \sqcup V_0 \sqcup (\sqcup_{i \in \Bar{\tau}(n)}E_i),t,*)$ over $I = \{-1,0\} \cup \Bar{\tau}(n)$ as follows:
\begin{itemize}
    \item $V_{-1}$ and $V_0$ are a bi-partition of the vertices of $V$ and $E_i$ is the set of edges of an inscribed regular polygon with edges of length $k$,
    \item $t(X_i) = i$,
    \item Let $v \in X_{-1} \cup X_0$ and $e \in X_i$ with $i \in \Bar{\tau}(n)$, then $e = \{v_1,v_2\}$ with $v_1,v_2 \in V$ and we set $v * e$ if and only if $v = v_1$ or $v = v_2$,
    \item For $i,j \in I \setminus \{-1,0\}$, we set that every elements $e_i\in X_i$ and $e_j \in X_j$ are incident as long as $i \neq j$.
    \item Every vertex of $V_0$ is incident to every vertex of $V_{-1}$.
\end{itemize}


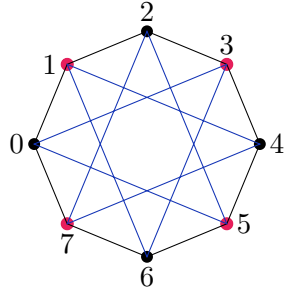
\begin{figure}
\begin{center}
\begin{tikzpicture}[scale = 1.5]
\label{D16}

\foreach \i in {1,...,8}
    \fill (\i*360/8:1) coordinate (8\i) circle(1.5 pt);
\foreach \i in {1,3,5,7}
    \filldraw [red] (\i*360/8:1)   circle(1.5 pt);
\draw (81) -- (82) -- (83) -- (84) -- (85) -- (86) -- (87)--(88)--(81);
\draw [blue] (81) -- (84) -- (87) -- (82)-- (85) -- (88)--(83)--(86)--(81);

\filldraw[black] (84) circle (0pt)  node[anchor=east]{$0$};
\filldraw[black] (83) circle (0pt)  node[anchor=east]{$1$};
\filldraw[black] (82) circle (0pt)  node[anchor=south]{$2$};
\filldraw[black] (81) circle (0pt)  node[anchor=south]{$3$};
\filldraw[black] (88) circle (0pt)  node[anchor=west]{$4$};
\filldraw[black] (87) circle (0pt)  node[anchor=west]{$5$};
\filldraw[black] (86) circle (0pt)  node[anchor=north]{$6$};
\filldraw[black] (85) circle (0pt)  node[anchor=north]{$7$};

\end{tikzpicture}
\caption{The geometry $\Gamma(D_{16})$. Since $n = 8$ is even, we separate the vertices two sets, corresponding to even and odd labels.}   
\end{center}
\end{figure}

\begin{theorem}\label{d2n}
The incidence system $\Gamma(D_{2n})$ is an incidence geometric representation for $D_{2n}$ for $n>3$ and it is a weak incidence geometric representation for $D_{2n}$ for $n = 3$.
\end{theorem}
\begin{proof}
    For $n= 3$, $\Gamma$ is the geometry of a triangle. In this case, we have that $D_6 \cong \Out(D_6)$ but $\Gamma$ has a duality. Therefore, we have that $\Aut_I(\Gamma) \cong D_6$ but $\Aut(\Gamma)$ is twice as large as $\Aut_I(\Gamma)$.

    Suppose now that $n >3$.
    We first explicitly define some automorphisms and correlations of $\Gamma = \Gamma(D_{2n})$. Label the vertices in $V$ by $0,1,\cdots, n-1$ in a cyclic fashion, so that we can identify $V$ with $\ZZ_n$. Then, an element $e \in E_i$ will be an unordered pair $\{x,y\}$ with $|x - y| = i \pmod n$. For any $x \in V$ and any $i \in \tau(n)$, define $\alpha_i(x) = x +i \pmod n$ and $\mu_i(x) = x \cdot i \pmod n$. Each of these functions $\alpha_i$ and $\mu_i$ induces a correlation of $\Gamma$. The functions $\mu_{1}$ is the identity and $\mu_{-1}$ is also type preserving as it corresponds to an axial symmetry fixing the vertex labeled $0$. The functions $m_{i}$ for $i \neq -1,1$ are never type preserving, as they send the set $E_j$ to $E_{ij}$.
    The $\alpha_i$ correspond to rotations of $P_n$. Hence, when $n$ is odd, all of the $\alpha_i$ are type preserving while when $n$ is even, $a_i$ is type preserving if and only if $i$ is even. Indeed, when $i$ is odd, $\alpha_i$ will exchange $V_{-1}$ and $V_0$.
    
    The subgroup $G$ of correlations of $\Gamma(D_{2n})$ generated by the functions $\alpha_i$ and $\mu_{i}$ is easily checked to be isomorphic to $\ZZ_n \rtimes \ZZ_{\phi(n)}$, where the $\ZZ_n$ component is the the group contained all the $\alpha_i$ and the $\ZZ_{\phi(n)}$ is the group containing all the $\mu_i$'s. The type preserving subgroup $H$ of $G$ is then a dihedral group of order $2n$ when $n$ is odd, generated by the $\alpha_i$ and $\mu_{-1}$, and a dihedral group of order $n$ when $n$ is even, generated by the $\alpha_i$ with $i$ even and $\mu_{-1}$. 

    Putting all the above together, we have so far shown that $\Gamma(D_{2n})$ is a weak incidence geometric representation for $D_{2n}$. To show that it is, in fact, an incidence geometric representation, we now need to show that the group $G$ defined above is $\Aut(\Gamma(D_{2n}))$. We will achieve this by proving that $|\Aut(\Gamma(D_{2n}))| = n \phi(n)$. There are $n$ choices for the image or a vertex $v$, as $\Aut(\Gamma(D_{2n}))$ is transitive on the vertices. Suppose thus that $v$ is fixed. There are $2 |I| = \phi(n)$ edges incident to $v$ and the stabilizer of $v$ in $\Aut(\Gamma(D_{2n}))$ acts transitively on these. Suppose now that $v$ and an edge $e =\{v,w\} \in E_i$ are fixed. All the other vertices of $P_n$ must then also be fixed. That is easily checked by looking at incidence inside of the regular inscribed polygon corresponding to $E_i$. Hence, the stabilizer of a pair $\{v,e\}$ is the identity and $|\Aut(\Gamma(D_{2n}))|$ is equal to $|V|$ multiplied by the number of edges incident to a vertex $v$ in $\Gamma(D_{2n})$, which is equal to $\phi(n)$ as required.
\end{proof}
\section{Symmetric and alternating groups}
The following theorem gives an incidence geometric representation for almost all groups $S_n$.
\begin{theorem}\label{sym}
    If $n \neq 3,6$, the complete graph $K_n$ on $n$ vertices is an incidence geometric representation for $\Sym(n)$.
\end{theorem}
\begin{proof}
    When $n \neq 6$, we have that $\Aut(\Sym(n)) = \Sym(n)$. Hence, we look for geometries that have $\Sym(n)$ has type preserving group and have no correlations. The graph $K_n$ is such a geometry, except when $n =3$, since, in that case, there is a duality that exchanges points and edges.
\end{proof}

For $n=3$, Theorem~\ref{d2n} gives a weak incidence geometric representation of $S_3\cong D_6$, 
and for $n=6$, the generalized quadrangle $GQ(2,2)$ is a weak incidence geometric representation of $S_6$. The generalized quadrangle $GQ(2,2)$ is a geometry of rank two whose group of type-preserving automorphisms is $S_6$ and whose group of correlations is $\Aut(S_6)\cong \PGammaL(2,9)$. It can be constructed by taking as points the 15 transpositions of $S_6$ acting on 6 points, and as lines the 15 3-transpositions of $S_6$ and defining incidence as inclusion, that is, a transposition is incident to a 3-transposition if the transposition is one of the three transpositions of the 3-transposition.


For the alternating groups, we refer to~\cite{atlasRWPRI} where examples of incidence geometric representations for $A_5$, $A_6$ and $A_7$ can be found.
However, at present we do not have examples for the alternating groups in general. 

\section{Automorphism groups of polyhedra}
In this section, we find incidence geometric representations for the automorphism groups of the five platonic solids. There are only three groups to consider, as a polyhedron has the same automorphism group as its dual. These groups are $\Sym(4)$, that is the group of the tetrahedron, $\Sym(4) \times C_2$, that is the group of the cube and the octahedron, and $A_5 \times C_2$, that is the group of the icosahedron and the dodecahedron.
The case of $\Sym(4)$ is already handled by Theorem~\ref{sym}.

\subsection{The cube}
Let $G = \Sym(4) \times C_2$. Then $\Aut(G) = G$ and $\Inn(G) = \Sym(4)$.

Let $\GG$ be a graph obtained as the $\{0,1\}$-truncation of the geometry of the cube and let $P= \{P_1,P_2\}$ be a proper 2-coloring of the vertices of $\GG$. We define a geometry with $4$ types as follows: the elements of type $1$ are the vertices in $P_1$, the elements of type $2$ are the vertices in $P_2$ and the elements of type $3$ and $4$ are the edges and faces of the cube. Incidence is induced from the cube.

\begin{theorem}
    $\Gamma(\Sym(4) \times C_2)$ is an incidence geometric representation for $\Sym(4) \times C_2$.
\end{theorem}
\begin{proof}
    First, we need to clarify the isomorphism between $G$, the automorphism group of the cube, and $\Sym(4) \times C_2$. Indeed, while the $C_2$ component must necessarily be generated by the central involution of the cube, there are two different subgroups of $G$, both isomorphic to $\Sym(4)$, that can be used in this decomposition. One of them is the rotational subgroup of the cube, which still acts transitively on the vertices, and the other is the group that fixes setwise the two tetrahedra inscribed in the cube, and has thus two orbits on the vertices. When we identify $G$ with $\Sym(4) \times C_2$, we from now on assume that the $\Sym(4)$ component is the group having two orbits on the vertices. 
    Now, note that each element of $G$ induces an automorphism of $\Gamma$, as the bipartition $P$ must be preserved by automorphisms of the cube. The central involution $\alpha$ of $G$ exchanges the sets $P_1$ and $P_2$, and thus is not type preserving. The group $G/\langle \alpha \rangle$ is type preserving, and is precisely $\Inn(G)$. Moreover, the map $G \to \Aut(\Gamma)$ is an injective homomorphism.
    It only remains to show that there are no other automorphisms of $\Gamma$. Any automorphism of $\Gamma$ would induce an automorphism of the cube, by considering elements of type $1$ and $2$ to both be vertices. Hence, there can be no more.
\end{proof}
\subsection{The dodecahedron}
Let $G = A_5 \times C_2$. Then $\Aut(G) = \Sym(5)$ and $\Inn(G) = A_5$.
Note that the group $G$ is the automorphism group of the dodecahedron and that the central involution sends each point to its antipodal. Let $\Gamma$ be the rank three geometry obtained from taking the quotient of the dodecahedron by its central involution: the hemidodecahedron. 
The hemidodecahedron is a projective polyhedron with ten vertices, fifteen edges and six pentagonal faces. The graph $\GG$ obtained as the $\{0,1\}$-truncation of $\Gamma$ (the $1$-skeleton) is the Petersen graph. While the faces in any polyhedron correspond to the closed paths in the $1$-skeleton obtained by always taking left at each vertex, the so-called Petrie polygons are the closed paths obtained by alternatively taking left and right. The Petersen graph has a total of twelve $5$-cycles. Six of them corresponds to the faces of $\Gamma$ and the other six are the Petrie polygons of $\Gamma$. 

We construct a geometry $\Gamma'$ as follows: the elements of type $0,1,2,3$ are the ten vertices, the fifteen edges, the six pentagonal faces and the six Petrie polygons of $\Gamma$. The $\{0,1,2\}$-truncation of $\Gamma'$ is $\Gamma$ and the $\{0,1,3\}$-truncation of $\Gamma'$ is also isomorphic to $\Gamma$, because the hemidodecahedron is self-Petrie dual; it is isomorphic to the geometry obtained by changing its faces for its Petrie polygons.


\begin{theorem}
    The geometry $\Gamma'$ is an incidence geometric representation for $A_5 \times C_2$.
\end{theorem}
\begin{proof}
    Let $G = A_5 \times C_2$. We have a natural map $\varphi_1 \colon \Inn(G) \to \Aut_I(\Gamma')$. Indeed, $\Inn(G)$ is precisely $G$ quotiented by its central involution, and is thus the group of automorphism of $\Gamma$. The action on the elements of type $3$ of $\Gamma'$ is induced by the action on the vertices. It is easy to see that $\varphi_1$ must be injective, hence an isomorphism. 

    The Petrie duality of $\Gamma$ induces a correlation of $\Gamma'$ of order $2$ that swaps the elements of type $2$ and $3$, that is, an order two element $\rho$ of $\Aut(\Gamma')\setminus \Aut_{I}(\Gamma')$. The action of $\rho$ on $\Aut_{I}(\Gamma')$ does clearly not induce an inner automorphism of $\Aut_{I}(\Gamma')$, therefore the extension of $\Aut_{I}(\Gamma')$ with $\rho$ is split but cannot be a direct product, and is thus isomorphic to $S_5$ (there are only three possible extensions of $A_5$ by $C_2$). Counting the number of elements of the other types shows that there are no other correlations than in this extension, in other words $\Aut(\Gamma')\cong S_5$.  
\end{proof}

Note that $\Gamma'$ is also an incidence geometric representation for $A_5$. Indeed, $\Inn(A_5) = \Inn(A_5 \times C_2) = A_5$ and $\Aut(A_5) = \Aut(A_5 \times C_2) = \Sym(5)$.

\section{Field automorphisms and Cross-ratios}

In this section, we give a general construction to model collineations of projective spaces as correlations of an incidence system in such a way that projectivities are type preserving while collineations induced by field automorphisms are not. This will give a general framework to find incidence geometric representations for many subgroups of $\PGL(n,\mathbf{K})$, for $n \geq 3$ and any field $\mathbf{K}$. These representations will only model the outer automorphisms induced by field automorphisms and the ones induced by the classical duality of projective spaces. If the group we are interested in has more outer automorphisms, the construction should be refined to take those into account if we want to obtain incidence geometric representation instead of a weak incidence geometric representation.

\subsection{The projective group $PGL(n,\mathbb{K})$}

We first handle the case of $\PGL(n,\mathbf{K})$, where $n \geq 3$ and $\mathbf{K}$ a field extension of some base field $\mathbf{F}$. Let $\mathbf{P} =\PG(n-1, \mathbf{K})$ be the projective space of (projective) dimension $n-1$ over the field $\mathbf{K}$, seen as an incidence geometry of rank $n-1$. For a quadruple $(p_1,p_2,p_3,p_4)$ of collinear points of $\PG(n-1, \mathbf{K})$, we denote by $[p_1:p_2; p_3:p_4]$ their cross-ratio. Let $\Gal(\mathbf{K}/\mathbf{F})$ be the Galois group of the extension $\mathbf{K}/\mathbf{F}$. We will consider this group $\Gal(\mathbf{K}/\mathbf{F})$ as not only acting on the field $\mathbf{K}$, but also on $\PG(n-1, \mathbf{K})$. The group $\PGL(n,\mathbf{K})$ is in fact the subgroup of $\Aut_I(\mathbf{P})$ that preserves this cross-ratio. The full group $\Aut_I(\mathbf{P})$ of type preserving automorphisms of $\mathbf{P}$ is usually denoted by $\PGammaL(n,\mathbf{K})$ and is isomorphic to $\PGL(n,\mathbf{K}) \rtimes \Gal(\mathbf{K}/\mathbf{F})$. The group $\Aut(\mathbf{P})$ of all correlations of $\mathbf{P}$ is then isomorphic to $\Aut(\PGL(n,\mathbf{K}))$ and is an extension of $\PGammaL(n,\mathbf{K})$ by the duality that exchanges points and hyperplanes of $\mathbf{K}$.
 We now construct a geometry $\Gamma = \Gamma(\PGL(n,\mathbf{K})) =(X,*,\tau,I)$ as follows.

\begin{itemize}
    \item The set $X$ of elements of $\Gamma$ is the disjoint union of the elements of $\mathbf{P}$, and the sets $Q_\lambda = \{ (p_1,p_2,p_3,p_4) \mid p_1,p_2,p_3,p_4 \text{ collinear and } [p_1:p_2; p_3:p_4] = \lambda \}$ for all $\lambda \in \mathbf{K} \setminus \mathbf{F}$.
    \item The type set $I = \{0,1, \cdots, n-2\} \sqcup (\mathbf{K} \setminus \mathbf{F})$.
    \item The type function $\tau$ sends elements of $\mathbf{P}$ to their projective dimension and elements of $Q_\lambda$ to $\lambda$.
    \item Incidence between elements of $\mathbf{P}$ is the same as in $\mathbf{P}$. An element $x \in \mathbf{P}$ is incident to a quadruple $(p_1,p_2,p_3,p_4) \in Q_\lambda$ if and only if $x = p_i$ for some $i = 1,2,3,4$ for some $p_i$ or $\{p_1,p_2,p_3,p_4\} \subset x$. Finally, any quadruple $(p_1,p_2,p_3,p_4) \in Q_{\lambda_1}$ is incident to any quadruple $(q_1,q_2,q_3,q_4) \in Q_{\lambda_2}$ as long as $\lambda_1 \neq \lambda_2$.
\end{itemize}

Notice that $\mathbf{P}$ is isomorphic to the truncation of $\Gamma$ of type $\{0,1,\cdots n-2\}$. We could have considered all quadruples of collinear points, even those whose cross-ratios have values in $\mathbf{F}$. We instead chose to only work with cross-ratios whose values are in $\mathbf{K} \setminus \mathbf{F}$ so that when $\mathbf{K} = \mathbf{F}$, we have that $\Gamma$ is isomorphic to  $\mathbf{P}$, which in that case is already a incidence geometric representation for $\PGL(n, \mathbf{F})$. We now show that $\Gamma$ is always an incidence geometric representation for $\PGL(n,\mathbf{K})$. Then, we will show how to reduce the numbers of types of $\Gamma$ when the extension $\mathbf{K}/\mathbf{F}$ is simple.


\begin{theorem}\label{thm:PGL}
    The incidence system $\Gamma(\PGL(n,\mathbf{K}))$ is an incidence geometric representation for $\PGL(n,\mathbf{K})$.
\end{theorem}

\begin{proof}
    We have that $\Inn(\PGL(n,\mathbf{K})) \cong \PGL(n,\mathbf{K})$ and $\Aut(\PGL(n,\mathbf{K})) \cong \Aut(\mathbf{P})$, which is an extension of $\PGL(n,\mathbf{K})$ by $\Gal(\mathbf{K}/\mathbf{F})$ and by the duality of $\mathbf{P}$. We must thus show that $\Aut(\Gamma) \cong \Aut(\mathbf{P})$ and $\Aut_I(\Gamma) \cong \PGL(n,\mathbf{K})$.
    
    Let $\varphi \in \Aut(\Gamma)$. Then, $\varphi$ must either send elements of type $0$ to elements of type $0$ or exchange elements of type $0$ and elements of type $n-2$. Indeed, it should be clear that elements of $\mathbf{P}$ cannot be sent to elements of $Q_\lambda$ for some $\lambda$. For example, an element $x \in Q_\lambda$ is incident to a unique element of type $1$ (the line that contains $x$), while elements of $\mathbf{P}$ are clearly incident to multiple elements of every type. Therefore, $\varphi$ acts on the $\{0,1, \cdots n-2\}$-truncation of $\Gamma$, which is isomorphic to $\mathbf{P}$. Let $f$ be the restriction of $\varphi$ to this $\{0,1, \cdots, n-{\color{red}2}\}$-truncation of $\Gamma$. We must have that $f$ is a collineation of $\PG(n-1, \mathbf{K})$, so that $f \in \PGammaL(n,\mathbf{K})$ if $f$ is type preserving and $f \in \Aut(\PGL(n,\mathbf{K}))/ \PGammaL(n,\mathbf{K})$ if $f$ exchanges points and hyperplanes. We claim that $f$ determines $\varphi$ entirely. For simplicity, we assume that $f$ is type preserving, but the general case is identical. Let $x =(p_1,p_2,p_3,p_4)$ be any element of $Q_\lambda$ for some $\lambda \in \mathbf{K} \setminus \mathbf{F}$. Then the $0$-elements of $\Gamma$ incident to $x$ are precisely $p_1,p_2,p_3,p_4$. It must then be that $\varphi(x)$ is incident to $f(p_i)$, $i = 1,2,3,4$ and to no other $0$-elements. The unique element of $\Gamma$ that satisfies this is the quadruple $(f(p_1),f(p_2),f(p_3),f(p_4))$ which must the be $\varphi(x)$. Hence, since we showed that there is a unique way to extend the action of $f$ to the rest of the elements of $\Gamma$, we conclude that $\Aut(\Gamma) \cong \Aut(\PGL(n,\mathbf{K}))$. 

    It is straightforward to check that the subgroup of $\Aut(\Gamma)$ corresponding to $\PGL(n,\mathbf{K})$ is type preserving while $\varphi \in \Gal(\mathbf{K}/\mathbf{F})$ sends elements of $Q_\lambda$ to elements of $Q_{\varphi(\lambda)}$ and the duality exchanges elements of type $0$ and $n-2$. This concludes the proof.    
%
\end{proof}

Suppose now that $\mathbf{K}$ is a simple extension of $\mathbf{F}$, meaning that $\mathbf{K} = \mathbf{F[\alpha]}$ for some $\alpha \in \mathbf{K}$. Let $P(\alpha)$ be the minimal polynomial for $\alpha$ with respect to the extension $\mathbf{K}/\mathbf{F}$. Let $R$ be the set of roots of $P(\alpha)$. We can construct an incidence system that only takes into account cross-ratios with values in $R$. We define $\Gamma(\alpha)$ to be equal to $^J\Gamma$, the truncation of $\Gamma$ of type $J = \{0,1\} \cup R$. 

\begin{corollary}
    Let $\mathbf{K} = \mathbf{F}[\alpha]$ be a simple field extension. Then, the incidence system $\Gamma(\alpha)$ is an incidence geometric representation for
    $\PGL(n,\mathbf{K})$.
\end{corollary}
\begin{proof}
    The proof of Theorem \ref{thm:PGL} holds verbatim in this case also.
\end{proof}

For example, if the fields $\mathbf{K}$ and $\mathbf{F}$ are finite, the extension $\mathbf{K}/\mathbf{F}$ is always cyclic. If $\phi$ is the Frobenius automorphism of order $m$, we can then choose one $\lambda$ such that $\phi(\lambda) \neq \lambda$ and the set of roots of a minimal polynomial will be $Q_\lambda, Q_{\phi(\lambda)}, \cdots, Q_{\phi^{m-1}(\lambda)}$.

\subsection{Subgroups of $\PGL(n,\mathbf{K})$}\label{classicalGroups}

A similar construction can be used for a variety of subgroups $H$ of $\PGL(n,\mathbf{K})$, as long we can describe them as stabilizer of subsets of $\PG(n-1, \mathbf{K})$. In some cases, these subgroups have additional outer automorphisms that do not come from fields automorphisms, but the geometry $\Gamma$ that we construct is only expected to be an incidence geometric representation for $(H,H\rtimes Gal(\mathbf{K}/\mathbf{F}))$, not for $H$.
We give a short and nonexhaustive list of examples of such subgroups $H$ together with a quick description of the geometry $\Gamma$ that could be investigated but without going into all the details.
\begin{enumerate}
    \item Orthogonal groups $PGO(n,\mathbf{K})$ are stabilizers of nondegenerate conics $\OO$. We can thus restrict to considering interior points of $\OO$ and secants lines to $\OO$.
    \item $\PGL(2,q) \cong PGO(3,q)$ so it falls into the above setting.
    \item Suzuki groups are stabilizers of Suzuki-Tits ovoids $\OO$. We can then take $\mathcal{P}$ to be the set of points of $\mathbf{P}$, $\mathcal{L}$ to be the set of circles of $\OO$ (intersection of hyperplanes with $\OO$) and consider quadruples of points on tangent lines to $\OO$. 
\end{enumerate}

\section{Free groups}\label{freeGroups}

Let $F_n = \langle x_1, \cdots, x_n \rangle$ be a free group of rank $n$. Contrary to the previous examples, both $F_n$ and its outer automorphism group $\Out(F_n)$ are infinite. Finding an incidence geometric representation for $F_n$ would require to construct an incidence system $\Gamma$ with $\Aut(\Gamma) = \Aut(F_n)$ and $\Aut_I(\Gamma) = \Inn(F_n)$. Then $\Aut(\Gamma)/\Aut_I(\Gamma) \cong \Out(F_n)$, which is infinite, so our incidence system would have infinitely many types. As the theory of incidence geometries with an infinite number of types still is underdeveloped, this is beyond our reach at the moment. Hence, in the sequel we decide to focus on finite subgroups of $\Out(F_n)$. 

Observe that finite subgroups of $\Aut(F_n)$ and finite subgroups of $\Out(F_n)$ are essentially the same, since $F_n \cong \Inn(F_n)$ is torsion-free. A realization theorem says that finite subgroups of $\Out(F_n)$ are point stabilizers under the action of $\Out(F_n)$ on Outer Space \cite{Vogtmann2002}. Therefore, they are isometry groups of finite graphs whose vertices all have valency at least three and whose fundamental group is $F_n$. That implies that the largest finite subgroup of $\Aut(F_n)$ is the isometry group $K$ of a rose with $n$ petals, which is the automorphism group of the $n$-dimensional hypercube. It is isomorphic to $K = C_2^n \rtimes \Sym(n)$ (see~\cite{FiniteOut} for more details).
Concretely, the $\Sym(n)$ factor is generated by permutations of the generators $\{x_1, \cdots, x_n\}$ and the involuntary involutions of the $C_2^n$ factor correspond to $\varphi_i$ sending $x_i$ to $x_i^{-1}$ while fixing the other generators, for $i = 1, \cdots, n$.

Let us recall how to construct an incidence system from a group and some of its subgroups. Such incidence systems are called {\em coset incidence systems} as the elements of these geometries are cosets of the subgroups.

\begin{proposition}\cite{Tits1956}\label{tits}
Let $n$ be a positive integer
and $I:= \{0,\ldots ,r-1\}$ a finite set.
Let $G$ be a group together with a family of subgroups ($G_i$)$_{i \in I}$, $X$ the set consisting of all cosets $G_ig$, $g \in G$, $i \in I$ and $t : X \rightarrow I$ defined by $t(G_ig) = i$.
Define an incidence relation $*$ on $X\times X$ by :
\begin{center}
$G_ig_1 * G_jg_2$ iff $G_ig_1 \cap G_jg_2$ is non-empty in $G$.
\end{center}
Then the 4-tuple $\Gamma := (X, *, t, I)$ is an incidence system having a chamber.
Moreover, the group $G$ acts by right multiplication as an automorphism group on $\Gamma$.
Finally, the group $G$ is transitive on the flags of rank less than three.
\end{proposition}

The above proposition is sometimes referred to as Tits algorithm to construct incidence systems from groups.
We will denote the coset incidence system of $G$ over $(G_i)_{i\in I}$ obtained from the construction described in Proposition~\ref{tits} by $\Gamma(G; (G_i)_{i\in I})$.
Moreover, we set $G_J = \cap_{j\in J}G_j$.
The following theorem permits to check when a coset incidence system is a flag-transitive incidence geometry.

 \begin{theorem}\label{FTlee2}\cite[Theorem 1.8.10 (i) and (iv)]{BC2013}
Let $\Gamma$ be a coset incidence system of $G$ over $(G_i)_{i\in I}$. If $I$ is finite, the following statements are equivalent.
\begin{enumerate}
    \item $G$ is flag-transitive on $\Gamma$.
    \item for each $J\subseteq I$ and each $i\in I\setminus J$, we have $G_JG_i = \cap_{j\in J}(G_jG_i)$.
\end{enumerate}
\end{theorem}

    We now prove that any coset incidence system constructed from an independent set of elements of a free groups is necessarily a flag-transitive geometry.  In fact, free groups together with a free generating set form Artin-Tits systems. Using known results about intersections of parabolic subgroups of Artin-Tits systems (see \cite{ParisGodelle2012}, \cite{van1983homotopy}), it can be shown that the standard coset incidence geometry built from any Artin-Tits system is always flag-transitive and residually connected, similarly as the case of Coxeter systems. The details of this claim are explained in a forthcoming paper \cite{AmalgamsIncidence}.
\begin{theorem}\label{freeFT}
    Let $I$ be a finite set and let $\{g_i : i \in I\}$ be an independent set of elements of the free group $F_n$. For each $i\in I$, let $G_i := \langle g_j : j \in I\backslash \{i\} \rangle$.
    Let $G = \langle g_i : i \in I \rangle$.
    Let $\Gamma = (G, (G_i)_{i \in I})$ be the coset geometry constructed from $G$ and the subgroups $G_i$'s.  Then $\Gamma$ is flag-transitive.
\end{theorem}

\begin{proof}
In order to check that $\Gamma$ is a flag-transitive geometry, we use Theorem~\ref{FTlee2} (2).
Hence we need to check that for each $J\subseteq I$ and each $i\in I\setminus J$, we have $G_JG_i = \cap_{j\in J}(G_jG_i)$. 
Obviously $G_JG_i \subseteq \cap_{j\in J}(G_jG_i)$.
Suppose one of these equalities does not hold. Then there exists a $J\subseteq I$ and a $i\in I \setminus J$ such that $\cap_{j\in J}(G_jG_i)\setminus G_JG_i\neq \emptyset$. Choose an element $g \in \cap_{j\in J}(G_jG_i)\setminus G_JG_i$ and choose two distinct indices $j_1,j_2 \in J$. Then, since $g \in G_{j_1}G_i \setminus G_JG_i$, there must exist a reduced word $w_1$ for $g$ that does not contain $g_{j_1}$ but contains $g_{k_1}$ for some $k_1 
\in J \setminus j_1$. In the same way, there must exist another reduced word $w_2$ for $g$ that does not contain $g_{j_2}$ but contains $g_{k_2}$ for some $k_2 \in J \setminus j_2$. Our assumption on $w_1$ and $w_2$ guarantees that the two words cannot be equal. Since both $w_1$ and $w_2$ represent $g$ in $F_n$, we found a non-trivial relations between independent elements of a free group, a contradiction. Hence all these equalities are trivially satisfied and $\Gamma$ is flag-transitive. 
\end{proof}

The following theorem permits to check when a flag-transitive coset geometry is residually connected.
 \begin{theorem}\label{RClee2}\cite[Corollary 1.8.13 (i) and (ii)]{BC2013}
Suppose that $I$ is finite and that $\Gamma = \Gamma(G, (G_i)_{i\in I})$ is a geometry over $I$ on which $G$ acts flag transitively. The following two statements are equivalent.
\begin{enumerate}
    \item $\Gamma$ is a residually connected geometry.
    \item for each $J\subseteq I$ with $|I\setminus J| \geq 2$, $G_J = \langle G_{J\cup \{i\}} : i \in I\setminus J\rangle$.
\end{enumerate}
\end{theorem}

Let us now construct a geometry $\Gamma$ on which $K$ acts injectively as a group of correlations. Unfortunately, only a subgroup of $F_n$ will act as a type-preserving automorphism group of $\Gamma$.
For the sake of clarity, we detail the case where $n=2$ before discussing the general case.

Let $n = 2$ and let $G$ be the subgroup of $F_2 = \langle x,y \rangle$ generated by the following four words.
$$yxy^{-1}, y^{-1}xy,xyx^{-1} \text{ and } x^{-1}yx.$$ 
Define four subgroups $G_1$, $G_2$, $G_3$ and $G_4$ taking for each of them three of the above four generators. In other words, let $G_ 1 = \langle y^{-1}xy,xyx^{-1},x^{-1}yx \rangle$, $G_ 2 = \langle yxy^{-1},xyx^{-1},x^{-1}yx \rangle$, $G_ 3 = \langle yxy^{-1},y^{-1}xy,x^{-1}yx \rangle$ and $G_ 4 = \langle yxy^{-1},y^{-1}xy,xyx^{-1} \rangle$.

\begin{theorem}\label{f2}
    The coset geometry $\Gamma = (G,(G_1,G_2,G_3,G_4))$ is residually connected and flag-transitive. Moreover $\Aut(\Gamma)/\Aut_I(\Gamma) \cong D_4$. Hence, $\Gamma$ is a weak incidence geometric representation for $G$. 
\end{theorem}

\begin{proof}
By Theorem~\ref{freeFT}, $\Gamma$ is flag-transitive.

Let $K =\langle \phi_1, \phi_2 \rangle$ where
$\phi_1 : F_2 \rightarrow F_2$ is an automorphism of $F_2$ given by $\phi_1(x) = x^{-1}$ and $\phi_1(y) = y$ and
$\phi_2 : F_2 \rightarrow F_2$ is an automorphism of $F_2$ given by $\phi_2(x) = y$ and $\phi_2(y) = x$.
We have that $K\cong D_4$ and $K\leq \Aut(F_2)$.
    It is straightforward to check that $K$ acts on $\{G_1,G_2,G_3,G_4\}$ in the same way that $D_4$ acts on the vertices of a square. Hence $\Aut(\Gamma)/\Aut_I(\Gamma) \cong D_4$.
    
    The group $G$ can be seen as the fundamental group of the graph $\mathcal{G}$ in Figure \ref{fig:graphs}, where the identification is suggested by the labeling of the edges of the graph. It is then clear that the subgroup $G_1$ corresponds to the subgroup of $G$ containing all the loops of $\mathcal{G}$ that never go into the bottom branch. Similarly, each $G_i$ correspond to the subgroups of all loops never going to one of the four branches. Therefore, it is apparent that $G_i \cap G_j = G_{i,j}$ for every distinct $i,j = 1,2,3,4$. The same holds for triple intersections and also $G_1 \cap G_2 \cap G_3 \cap G_4 = \{e\}$. It is now straightforward to check that $\Gamma$ is residually connected thanks to Theorem~\ref{RClee2}.
\end{proof} 

The key in the construction of the previous example is to take each generator $x_1, \ldots, x_n$ of $F_n$ twice, conjugating it by another generator or the inverse of that other generator.
We thus get $n\cdot (n-1)\cdot 2$ words, of the form $x_jx_ix_j^{-1}$ or $x_j^{-1}x_ix_j$.
Thus for an arbitrary $n$, we construct a graph $\mathcal{G}_n$ generalizing the one in Figure \ref{fig:graphs}. Recall that $F_n = \langle x_1 , \cdots, x_n\rangle $. The graph will live inside of $\mathbb{R}^n$. It is made of the $n$-coordinate axis, understood as graphs with vertices at integer coordinates. Moreover, to each vertex of $\mathcal{G}_n$ we attach a rose with $n-1$ petals. Each edge going out of the base vertex $(0,0,\cdots,0)$ is labeled by $x_i$, where the $i$ is chosen according to which coordinate axis the edge belongs to. We continue labeling edges by $x_i$ in each axis. The remaining non-labeled edges are all loops at some vertex $p$. We label them individually with some $x_j$ different from the $x_i$ corresponding to the axis they are attached to.
The group $G$ is then the fundamental group of $\mathcal{G}_n$. It is generated by the $2n(n-1)$ elements described above, namely
$$G:= \langle x_jx_ix_j^{-1},x_j^{-1}x_ix_j | i,j \in \{1, \ldots, n\}, i\neq j\rangle.$$
The maximal parabolic subgroups $G_i$ are then the $2n(n-1)$ subgroups of $G$ generated by all but one of these generators.

Now the group $K$ is the group generated by the following three automorphisms of $F_n$ that we give by precising the images of the generators $x_1, \ldots, x_n$.

$$\phi_1 : F_n\rightarrow F_n : (x_1, \ldots, x_n) \mapsto (x_1^{-1}, x_2, \ldots, x_n)$$
$$\phi_\rho : F_n\rightarrow F_n : (x_1, \ldots, x_n) \mapsto (x_2, \ldots, x_n, x_1)$$
$$\phi_\tau : F_n\rightarrow F_n : (x_1, \ldots, x_n) \mapsto (x_2, x_1, x_3 \ldots, x_n)$$

The group $K := \langle \phi_1, \phi_\rho, \phi_\tau \rangle$ is isomorphic to $C_2^n \rtimes \Sym(n)$ where the group $C_2^n$ acts on the generators $x_i$ individually, swapping them with their inverses and the group $\Sym(n)$ is generated by $\phi_\rho$ and $\phi_\tau$ that swap the generators.
\begin{theorem}
    The coset geometry $\Gamma = (G,(G_i)_{i \in I})$ is residually connected and flag-transitive. Moreover, $K$ acts as correlations of $\Gamma$. Hence, $\Gamma$ is a weak incidence geometric representation for $G$. 
\end{theorem}

\begin{proof} The fact that $K$ is as desired and flag-transitivity are proved in the same way as in Theorem~\ref{f2}. For residual connectedness, the proof is identical to the one given in the proof of Theorem~\ref{f2}.
\end{proof}

\begin{figure}
\centering
\begin{tikzpicture}

\foreach \i in {-5,...,5}
{
\filldraw[black] (\i,0) circle (1pt)  node[anchor=west]{};
\filldraw[black] (0,\i) circle (1pt)  node[anchor=west]{};
}

\foreach \i in {-5, ..., -1}{
\draw[->](\i+1,0) -- (\i+0.5,0) node[anchor=south]{$x$};
\node at (\i,0.8){$y$};
\draw[->](0,\i+1) -- (0,\i +0.5) node[anchor=east]{$y$};
\node at (0.8,\i){$x$};
\draw (\i,0.3) circle [radius=0.3];
\draw (0.3,\i) circle [radius=0.3];
}

\foreach \i in {1, ..., 5}{
\draw[->](\i-1,0) -- (\i-0.5,0) node[anchor=south]{$x$};
\node at (\i,0.8){$y$};
\draw[->](0,\i-1) -- (0,\i -0.5) node[anchor=east]{$y$};
\node at (0.8,\i){$x$};
\draw (\i,0.3) circle [radius=0.3];
\draw (0.3,\i) circle [radius=0.3];
}

\draw (-5.5,0) -- (5.5,0) ;
\draw (0,5.5) -- (0,-5.5) ;

\end{tikzpicture}
   \caption{The cover of the rose with two petals corresponding to the subgroup $G$ of $F_2$}
\label{fig:graphs}
\end{figure}
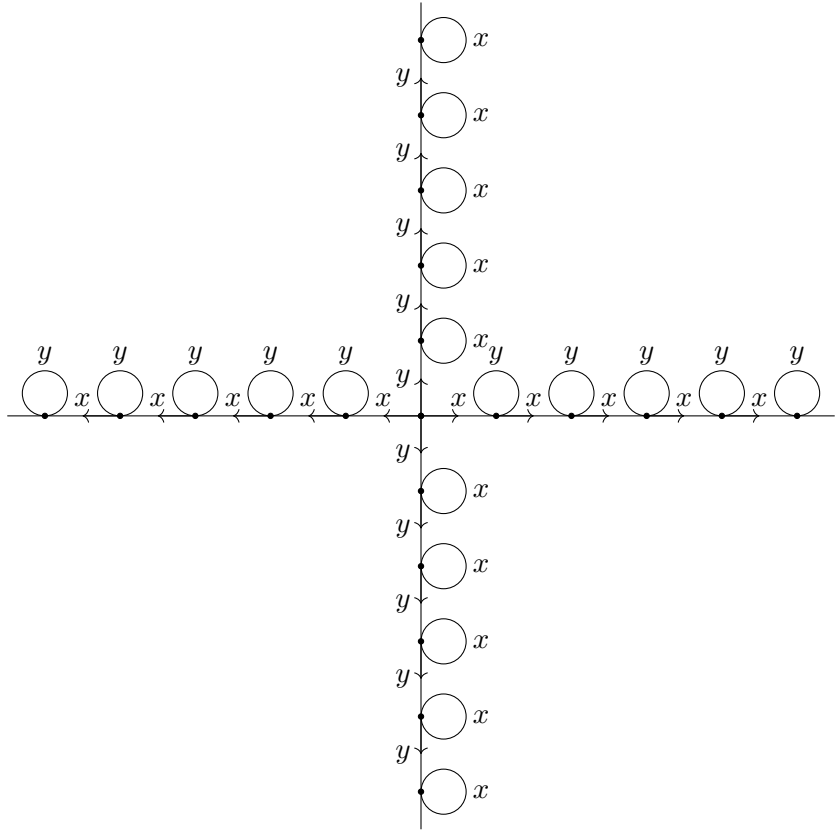

\section{Future work}
In Section~\ref{freeGroups}, we restricted ourselves to finite subgroups of the outer automorphism group as, otherwise, we would need to construct coset geometries with infinitely many types. The theory of such coset geometries has not yet, to our knowledge, been developed. It would be interesting and needed for us to push in that direction.

In future work we intend to study in more details incidence geometric representations for all finite almost simple groups and in particular for classical groups mentioned in Section~\ref{classicalGroups}.

In another direction, it would be natural to study whether, for a fixed group $G$, one can find a geometric representation $\Gamma$ with additional properties, such as flag-transitivity, residual connectedness, thinness, and so on. In doing so, some care is necessary. For example, it is not hard to show that if $\Gamma$ is flag-transitive and thin, then $\Aut(\Gamma)$ is a semi-direct product of $\Aut_I(\Gamma)$ with the subgroup of $\Sym(I)$ representing the action of $\Aut(\Gamma)$ on the type set $I$. Therefore, only groups $G$ such that $\Aut(G) \cong \Inn(G) \rtimes \Out(G)$ have a chance to admit flag-transitive and thin geometric representations. Other relations between algebraic properties of $\Aut(G)$ and $\Inn(G)$ and geometric properties of their geometric representations likely exist, and would deserve to be studied further.
\bibliographystyle{abbrv} 
\bibliography{refs}
\end{document}